\documentclass[reqno,12pt]{amsart}
\usepackage{amscd,amsfonts,amssymb}
\usepackage[T1,T2A]{fontenc}
\usepackage[cp1251]{inputenc}
\numberwithin{equation}{section}
\newtheorem{theorem}{Theorem}[section]
\newtheorem{lemma}{Lemma}[section]

\newtheorem{proposition}{Proposition}[section]
\theoremstyle{definition}

\theoremstyle{remark}

\author{S.\,M. Bakiev}
\address{Department of Differential Equations,
Faculty of Mechanics and Mathematics,
Mo\-s\-cow Lo\-mo\-no\-sov State University,
Vorobyovy Gory,
Moscow, 119992 Russia}
\email{pifagorgor@gmail.com}
\author{A.\,A. Kon'kov}
\address{Department of Differential Equations,
Faculty of Mechanics and Mathematics,
Mo\-s\-cow Lo\-mo\-no\-sov State University,
Vorobyovy Gory,
Moscow, 119992 Russia}
\email{konkov@mech.math.msu.su}

\title[On the existence of solutions]{
On the existence of solutions of the Dirichlet problem for $p$-Laplacian on Riemannian manifolds
}

\thanks{}
\date{}

\begin{document}

\begin{abstract}
We obtain a criterion for the existence of solutions of the problem
$$
	\Delta_p u = 0
	\quad
	\mbox{in } M \setminus \partial M,
	\quad
	\left.
		u
	\right|_{
		\partial M
	}
	=
	h,
$$
with the bounded Dirichlet integral, where $M$ is an oriented complete Riemannian manifold with boundary and $h \in W_{p, loc}^1 (M)$, $p > 1$.
\end{abstract}

\maketitle

\section{Introduction}
\label{sec1}

Let $M$ be an oriented complete Riemannian manifold with boundary.
We consider solutions of the problem
\begin{equation}
	\Delta_p u = 0
	\quad
	\mbox{in } M \setminus \partial M,
	\label{1.1}
\end{equation}
\begin{equation}
	\left.
		u
	\right|_{
		\partial M
	}
	=
	h,
	\label{1.2}
\end{equation}
where
$
	\Delta_p u
	=
	\nabla_i
	(
		g^{ij}
		|\nabla u|^{p - 2}
		\nabla_j u
	)
$
is the $p$-Laplacian and $h \in W_{p, loc}^1 (M)$, $p > 1$.

As a condition at infinity, we assume that the Dirichlet integral is bounded, i.e.
\begin{equation}
	\int_M
	|\nabla u|^p
	\,
	dV
	<
	\infty.
	\label{1.3}
\end{equation}

As is customary, by $g_{ij}$ we denote the metric tensor consistent with the Rie\-ma\-ni\-an connection and by $g^{ij}$ we denote the tensor dual to the metric one.
In so doing, $|\nabla u| = (g^{ij} \nabla_i u \nabla_j u)^{1/2}$.
As in~\cite{LU}, by $W_{p, loc}^1 (\omega)$, where $\omega \subset M$ is an open set, we mean the space of measurable functions belonging to $W_p^1 (\omega' \cap \omega)$ for any open set $\omega' \subset M$ with compact closure.
The space $L_{p, loc} (\omega)$ is defined analogously.

A function $u \in W_{p, loc}^1 (M)$ is called a solution of~\eqref{1.1} if
\begin{equation}
	\int_M
	g^{ij}
	|\nabla u|^{p - 2}
	\nabla_j u
	\nabla_i \varphi
	\,
	dV
	=
	0
	\label{1.6}
\end{equation}
for all $\varphi \in C_0^\infty (M \setminus \partial M)$, where $dV$ is the volume element of the manifold $M$.
In its turn, condition~\eqref{1.2} means that
$
	(u - h) \psi
	\in 
	{\stackrel{\rm \scriptscriptstyle o}{W}\!\!{}_p^1 (M \setminus \partial M)}
$
for all $\psi \in C_0^\infty (M)$.

Boundary value problems for differential equations in unbounded domains and on smooth manifolds have been studied by a number of authors~\cite{BK}--\cite{Kudryavtsev}, \cite{MPPMA2011}.
In the case where $M$ is a domain in ${\mathbb R}^n$ bounded by a surface of revolution, a criterion for the existence of solutions of~\eqref{1.1}--\eqref{1.3} was obtained in~\cite {MPPMA2011}.
However, the method used in~\cite {MPPMA2011} cannot be generalized to the case of an arbitrary Riemannian manifold.
Theorem~\ref{t2.1} proved in our article does not have this shortcoming.

Let $K \subset M$ be a compact set. We denote by $C_0^\infty (M, K)$ the set of functions from $C^\infty (M)$ that are equal to zero in a neighborhood of $K$.
In its turn, by ${\stackrel{\rm \scriptscriptstyle o}{W}\!\!{}_p^1 (\omega, K)}$, where $\omega$ is an open subset of $M$, we denote the closure of $C_0^\infty (M, K) \cap W_p^1 (\omega)$ in $W_p^1 (\omega)$.
By definition, a function $\varphi \in W_{p, loc}^1 (M)$ satisfies the condition
\begin{equation}
	\left.
		\varphi
	\right|_{
		K
	}
	=
	\psi,
	\label{1.4}
\end{equation}
where $\psi \in W_{p, loc}^1 (M)$, if
$
	\varphi - \psi 
	\in 
	{\stackrel{\rm \scriptscriptstyle o}{W}\!\!{}_p^1 (\omega, K)}
$
for some open set $\omega$ containing $K$.

\begin{proposition}\label{p1.1}
A function $u \in W_{p, loc}^1 (\Omega)$ satisfies~\eqref{1.2} if and only if
\begin{equation}
	\left.
		u
	\right|_{
		K
	}
	=
	h
	\label{p1.1.1}
\end{equation}
for any compact set $K \subset \partial M$.
\end{proposition}

\begin{proof}
At first, let~\eqref{1.2} hold and $K$ be a compact subset of~$\partial M$.
Take an open pre-compact set $\omega$ containing $K$ and a function $\psi \in C_0^\infty (M)$ such that
$$
	\left.
		\psi
	\right|_\omega
	=
	1.
$$
By~\eqref{1.2}, the function $(u - h) \psi$ belongs to the closure of $C_0^\infty (M \setminus \partial M)$ in the space $W_p^1 (M \setminus \partial M)$. Assuming that functions from ${C_0^\infty (M \setminus \partial M)}$ are extended by zero to $\partial M$, we obtain
$
	u - h
	\in 
	{\stackrel{\rm \scriptscriptstyle o}{W}\!\!{}_p^1 (\omega, K)}.
$

Now, assume that condition~\eqref{p1.1.1} is valid and let $\psi \in C_0^\infty (M)$.
We consider the compact set $K = \operatorname{supp} \psi \cap \partial M$.
In view of~\eqref{p1.1.1}, there exists an open set $\omega$ such that
$K \subset \omega$ and, moreover,
$
	u - h
	\in 
	{\stackrel{\rm \scriptscriptstyle o}{W}\!\!{}_p^1 (\omega, K)}
$
or, in other words, 
\begin{equation}
	\| u - h - \varphi_i \|_{
		W_p^1 (\omega)
	}
	\to
	0
	\quad
	\mbox{as } i \to \infty
	\label{pp1.1.1}
\end{equation}
for some sequence of functions
$\varphi_i \in C_0^\infty (M, K) \cap W_p^1 (\omega)$, $i = 1,2,\ldots$.
We denote $\tilde K = \operatorname{supp} \psi \setminus \omega$. 
Since $\tilde K$ is a compact set belonging to $M\setminus\partial M$, there is a function $\tau \in C_0^\infty (M \setminus \partial M)$ equal to one in a neighborhood of $\tilde K$.
It is easy to see that
$
	(1 - \tau) \psi \varphi_i 
	\in
	C_0^\infty (\omega \setminus \partial M),
$
$i = 1,2,\ldots$.
At the same time, by~\eqref{pp1.1.1}, we have
$$
	\| 
		(1 - \tau) \psi (u - h - \varphi_i)
	\|_{
		W_p^1 (M)
	}
	=
	\| 
		(1 - \tau) \psi (u - h - \varphi_i)
	\|_{
		W_p^1 (\omega)
	}
	\to
	0
	\quad
	\mbox{as } i \to \infty;
$$
therefore, one can assert that
$
	(1 - \tau) \psi (u - h) 
	\in 
	{\stackrel{\rm \scriptscriptstyle o}{W}\!\!{}_p^1 (M \setminus \partial M)}.
$
It is also obvious that
$
	\tau \psi (u - h) 
	\in 
	{\stackrel{\rm \scriptscriptstyle o}{W}\!\!{}_p^1 (M \setminus \partial M)}.
$
Thus, we obtain
$
	\psi (u - h) 
	=
	(1 - \tau) \psi (u - h) 
	+
	\tau \psi (u - h) 
	\in 
	{\stackrel{\rm \scriptscriptstyle o}{W}\!\!{}_p^1 (M \setminus \partial M)}.
$
\end{proof}

Let $\Omega$ be an open subset of $M$.
The capacity of a compact set $K \subset M$ associated with a function $\psi \in W_{p, loc}^1 (M)$ is defined as
$$
	\operatorname{cap}_\psi (K, \Omega)
	=
	\inf_\varphi
	\int_\Omega
	|\nabla \varphi|^p
	dV,
$$
where the infimum is taken over all functions
$
	\varphi
	\in
	{\stackrel{\rm \scriptscriptstyle o}{W}\!\!{}_p^1 (\Omega)}
$
for which~\eqref{1.4} is valid. In so doing, we assume that the functions from
${\stackrel{\rm \scriptscriptstyle o}{W}\!\!{}_p^1 (\Omega)}$ are
extended by zero beyond $\Omega$.
For an arbitrary closed set $E \subset M$, we put
$$
\operatorname{cap}_\psi (E, \Omega)
=
\sup_K
\operatorname{cap}_\psi (K, \Omega),
$$
where the supremum is taken over all compact sets $K \subset E$.
If $\Omega = M$, we write $\operatorname{cap}_\psi (K)$ instead of $\operatorname{cap}_\psi (K, M)$.
In the case of $\psi = 1$ and $p = 2$, the capacity $\operatorname{cap}_\psi (K)$ coincides with the well-known Wiener capacity~\cite{Landkof}.

It is not difficult to verify that the capacity introduced above has the following natural properties.

\begin{enumerate}
\item[(a)]
Let
$K_1 \subset K_2$ 
and
$\Omega_2 \subset \Omega_1$,
then
$$
	\operatorname{cap}_\psi (K_1,\Omega_1) \le \operatorname{cap}_\psi (K_2,\Omega_2).
$$

\item[(b)]
Suppose that $\lambda$ is a real number, then
$$
	\operatorname{cap}_{\lambda \psi} (K, \Omega) = |\lambda|^p \operatorname{cap}_\psi (K, \Omega).
$$

\item[(c)]
Let $\psi_1, \psi_2 \in W_{p, loc}^1 (M)$, then
$$
	\operatorname{cap}_{\psi_1 + \psi_2}^{1 / p} (K, \Omega) 
	\le
	\operatorname{cap}_{\psi_1}^{1 / p} (K, \Omega)
	+
	\operatorname{cap}_{\psi_2}^{1 / p} (K, \Omega).
$$
\end{enumerate}

We say that $u \in W_{p, loc}^1 (M)$ is a solution of~\eqref{1.1} under the condition
\begin{equation}
	\left.
		\frac{\partial u}{\partial \nu}
	\right|_{\partial M}
	=
	0
	\label{1.5}
\end{equation}
if the integral identity~\eqref{1.6} holds for all $\varphi \in C_0^\infty (M)$. 
The set of solutions of problem~\eqref{1.1}, \eqref{1.5} with bounded Dirichlet integral~\eqref{1.3} is denoted by $\mathfrak{H}$.

\section{Main result}
\begin{theorem}\label{t2.1}
Problem~\eqref{1.1}--\eqref{1.3} has a solution if and only if
\begin{equation}
	\operatorname{cap}_{h - w} (\partial M)
	<
	\infty
	\label{t2.1.1}
\end{equation}
for some $w \in \mathfrak{H}$.
\end{theorem}

The proof of Theorem~\ref{t2.1} is based on the following two lemmas known as Poincare's inequalities.

\begin{lemma}\label{l2.1}
Let $G \subset M$ be a pre-compact Lipschitz domain and $\omega $ be a subset of $G$ of non-zero measure. Then
$$
	\int_G
	|u|^p
	dV
	\le
	C
	\left(
		\int_G
		|\nabla u|^p
		dV
		+
		\left|
			\int_\omega
			u
			\,
			dV
		\right|^p
	\right)
$$
for all $u \in W_p^1 (G)$, where the constant $C > 0$ does not depend on $u$.
\end{lemma}

\begin{lemma}\label{l2.2}
Let $\omega \subset M$ be a pre-compact Lipschitz domain. Then
$$
	\int_\omega
	|\varphi - \alpha|^p
	\,
	dV
	\le
	C
	\int_\omega
	|\nabla \varphi|^p
	\,
	dV,
$$
for all $\varphi \in W_p^1 (\omega)$, where
$$
	\alpha
	=
	\frac{
		1
	}{
		\operatorname{mes} \omega
	}
	\int_\omega
	\varphi
	\,
	dV
$$
and the constant $C > 0$ does not depend on $\varphi$.
\end{lemma}

\begin{proof}[Proof of Theorem~$\ref{t2.1}$]
We show that the existence of a solution of~\eqref {1.1}--\eqref {1.3} implies the validity of~\eqref {t2.1.1}.
Consider a sequence of functions
$\varphi_i \in C_0^\infty (M)$, 
$i = 1,2,\ldots$, such that
$$
	\int_M
	|\nabla (u - \varphi_i)|^p
	dV
	\to
	\inf_{
		\varphi \in C_0^\infty (M)
	}
	\int_M
	|\nabla (u - \varphi)|^p
	dV
	\quad
	\mbox{as }
	i \to \infty.
$$
Since the sequence $\nabla \varphi_i $, $i = 1.2, \ldots $, is bounded in $L_p (M)$, there is a subsequence $\nabla \varphi_ {i_j}$, $j = 1,2,\ldots$, that converges weakly in $L_p (M)$ to some vector-function ${\mathbf r} \in L_p (M)$.
Let $R_m$ be the convex hull of the set $\{ \varphi_{i_j} \}_{j \ge m}$.
By Mazur's theorem, there exists a sequence $r_m \in R_m$, $m = 1,2,\ldots$, such that
\begin{equation}
	\| \nabla r_m - {\mathbf r} \|_{
		L_p (M)
	}
	\to
	0
	\quad
	\mbox{as }
	m \to \infty.
	\label{pt2.1.1}
\end{equation}
In view of the convexity of the functional
$$
	\varphi 
	\mapsto 
	\int_M
	|\nabla (u - \varphi)|^p
	dV,
	\quad
	\varphi
	\in
	{\stackrel{\rm \scriptscriptstyle o}{W}\!\!{}_p^1 (M)},
$$
we have
$$
	\int_M
	|\nabla (u - r_m)|^p
	dV
	\le
	\sup_{
		j \ge m
	} 
	\int_M
	|\nabla (u - \varphi_{i_j})|^p
	dV;
$$
therefore,
$$
	\int_M
	|\nabla (u - r_m)|^p
	dV
	\to
	\inf_{
		\varphi \in C_0^\infty (M)
	}
	\int_M
	|\nabla (u - \varphi)|^p
	dV
	\quad
	\mbox{as }
	m \to \infty.
$$
Let $\omega \subset M$ be a pre-compact Lipschitz domain. Denoting
$$
	\alpha_m
	=
	\frac{
		1
	}{
		\operatorname{mes} \omega
	}
	\int_\omega
	r_m
	\,
	dV,
$$
we obtain in accordance with Lemma~\ref{l2.2} that the sequence $r_m - \alpha_m$, $m  = 1,2,\ldots$,
is fundamental in $W_p^1 (\omega)$. By Lemma~\ref{l2.1}, this sequence is also fundamental in $W_p^1 (G)$ for any pre-compact Lipschitz domain $G \subset M$. 

At first, we assume that the sequence $\alpha_m$, $m = 1,2,\ldots$, is bounded. Extracting from it a convergent subsequence $\alpha_{i_j}$, $j = 1,2,\ldots$, we have that the sequence of the functions $r_{m_j}$, $j = 1,2,\ldots$, 
is fundamental in $W_p^1 (G)$ for any pre-compact Lipschitz domain $G \subset M$. 
Hence, there exists $v \in W_{p, loc}^1 (M)$ such that
$$
	\| r_{m_j} - v \|_{
		W_p^1 (G)
	}
	\to
	0
	\quad
	\mbox{as }
	j \to \infty
$$
for any pre-compact Lipschitz domain $G \subset M$. In view of~\eqref{pt2.1.1}, we have
$
	\nabla v = {\mathbf r};
$
therefore,
\begin{equation}
	\int_M
	|\nabla (u - v)|^p
	dV
	=
	\inf_{
		\varphi \in C_0^\infty (M)
	}
	\int_M
	|\nabla (u - \varphi)|^p
	dV.
	\label{pt2.1.2}
\end{equation}
Thus, by the variational principle, the function $w = u - v$ belongs to $\mathfrak{H}$.

Let us show the validity of inequality~\eqref{t2.1.1}.
Let $K \subset \partial \Omega$ be some compact set. It is easy to see that
\begin{equation}
	\left.
		v
	\right|_{
		K
	}
	=
	h - w.
	\label{pt2.1.3}
\end{equation}
Take a function $\tau \in C_0^\infty (M)$ equal to one in a neighborhood of $K$.
Putting
$
	\psi_j
	=
	\tau v
	+
	(1 - \tau) r_{m_j},
$
$
	j = 1,2,\ldots,
$
we obtain a sequence of functions from
${\stackrel{\rm \scriptscriptstyle o}{W}\!\!{}_p^1 (M)}$
satisfying the condition
$$
	\left.
		\psi_j
	\right|_{
		K
	}
	=
	h - w,
	\quad
	j = 1,2,\ldots.
$$
In so doing, we obviously have
\begin{align*}
	&
	\int_M
	|\nabla (v - \psi_j)|^p
	dV
	=
	\int_M
	|\nabla ((1 - \tau)(v - r_{m_j}))|^p
	dV
	\\
	&
	\quad
	{}
	\le
	2^p
	\int_{
		\operatorname{supp} \tau
	}
	|\nabla \tau(v - r_{m_j})|^p
	dV
	+
	2^p
	\int_M
	|(1 - \tau) \nabla (v - r_{m_j})|^p
	dV
	\to
	0
	\;
	\mbox{as } j \to \infty,
\end{align*}
whence it follows immediately that
\begin{equation}
	\operatorname{cap}_{h - w} (K)
	\le
	\lim_{j \to \infty}
	\int_M
	|\nabla \psi_j|^p
	dV
	=
	\int_M
	|\nabla v|^p
	dV.
	\label{pt2.1.4}
\end{equation}
In view of the arbitrariness of the compact set $K \subset \partial \Omega$, the last formula implies the estimate
\begin{equation}
	\operatorname{cap}_{h - w} (\partial M)
	\le
	\int_M
	|\nabla v|^p
	dV
	<
	\infty.
	\label{pt2.1.5}
\end{equation}

Now, assume that the sequence $\alpha_m$, $m = 1,2,\ldots$, is not bounded.
Without loss of generality, we can also assume that
$|\alpha_m| \to \infty$ as $m \to \infty$.
If this is not the case, then we replace $\alpha_m$, $m = 1,2,\ldots$, with a suitable subsequence.
Applying Lemma~\ref{l2.2}, we arrive at the inequality
$$
	\int_\omega
	|r_m - \alpha_m|^p
	\,
	dV
	\le
	C
	\int_\omega
	|\nabla r_m|^p
	\,
	dV
$$
for all $m = 1,2,\ldots$, where the constant $C > 0$ does not depend on $m$, whence we have
$$
	\int_\omega
	\left|
		\frac{r_m}{\alpha_m} 
		- 
		1
	\right|^p
	\,
	dV
	\le
	\frac{
		C
	}{
		|\alpha_m|^p
	}
	\int_\omega
	|\nabla r_m|^p
	dV
	\to
	0
	\quad
	\mbox{as } m \to \infty.
$$
For any positive integer $m$ we take a positive integer $s_m \ge m$ such that
\begin{equation}
	\int_\omega
	\left|
		\alpha_m
		-
		\frac{
			\alpha_m
			r_{s_m}
		}{
			\alpha_{s_m}
		}
	\right|^p
	dV
	=
	|\alpha_m|^p
	\int_\omega
	\left|
		\frac{
			r_{s_m}
		}{
			\alpha_{s_m}
		}
		- 
		1
	\right|^p
	dV
	<
	\frac{1}{2^m}
	\label{pt2.1.6}
\end{equation}
and
\begin{equation}
	\left| 
		\frac{
			\alpha_m
		}{
			\alpha_{s_m}
		}
	\right|
	<
	\frac{1}{2^m}.
	\label{pt2.1.7}
\end{equation}
Putting further
$$
	v_m 
	= 
	r_m 
	- 
	\frac{
		\alpha_m
		r_{s_m}
	}{
		\alpha_{s_m}
	},
	\quad
	m = 1,2,\ldots,
$$
we obtain
\begin{align*}
	\int_\omega
	|v_m - v_l|^p
	dV
	\le
	{}
	&
	2^p
	\int_\omega
	|r_m - r_l - \alpha_m + \alpha_l|^p
	dV
	\\
	&
	{}
	+
	2^p
	\int_\omega
	\left|
		\alpha_m
		-
		\frac{
			\alpha_m
			r_{s_m}
		}{
			\alpha_{s_m}
		}
		-
		\alpha_l
		+
		\frac{
			\alpha_l
			r_{s_l}
		}{
			\alpha_{s_l}
		}
	\right|^p
	dV,
	\quad
	m,l = 1,2,\ldots.
\end{align*}
By Lemma~\ref{l2.2}, the estimate
$$
	\int_\omega
	|r_m - r_l - \alpha_m + \alpha_l|^p
	dV
	\le
	C
	\int_\omega
	|\nabla (r_m - r_l)|^p
	dV,
	\quad
	m,l = 1,2,\ldots,
$$
is valid, where the constant $C > 0$ does not depend on $m$ and $l$.
At the same time, condition~\eqref{pt2.1.6} allows us to assert that
\begin{align*}
	&
	\int_\omega
	\left|
		\alpha_m
		-
		\frac{
			\alpha_m
			r_{s_m}
		}{
			\alpha_{s_m}
		}
		-
		\alpha_l
		+
		\frac{
			\alpha_l
			r_{s_l}
		}{
			\alpha_{s_l}
		}
	\right|^p
	dV
	\le
	2^p
	\int_\omega
	\left|
		\alpha_m
		-
		\frac{
			\alpha_m
			r_{s_m}
		}{
			\alpha_{s_m}
		}
	\right|^p
	dV
	\\
	&
	\qquad
	{}
	+
	2^p
	\int_\omega
	\left|
		\alpha_l
		-
		\frac{
			\alpha_l
			r_{s_l}
		}{
			\alpha_{s_l}
		}
	\right|^p
	dV
	<
	\frac{2^p}{2^m}
	+
	\frac{2^p}{2^l},
	\quad
	m,l = 1,2,\ldots.
\end{align*}
Hence, the sequence $v_m$, $m = 1,2,\ldots$, is fundamental in $L_p (\omega)$.
According to Lemma~\ref{l2.1}, this sequence is also fundamental in $W_p^1 (G)$ 
for any pre-compact Lipschitz domain $G \subset M$.
Let us denote by $v$ the limit of this sequence.
In view of~\eqref{pt2.1.1} and \eqref{pt2.1.7}, we have
$$
	\| \nabla v_m - {\mathbf r} \|_{
		L_p (M)
	}
	\to
	0
	\quad
	\mbox{as }
	m \to \infty;
$$
therefore, $v$ satisfies relation~\eqref{pt2.1.2} and in accordance with the variational principle the function $w = u - v$ belongs to $\mathfrak{H}$. 
In so doing, for any compact set $K \subset \partial M$ condition~\eqref{pt2.1.3} is obviously valid.
Thus, putting 
$
	\psi_j
	=
	\tau v
	+
	(1 - \tau) v_j,
$
$	
	j = 1,2,\ldots,
$
where $\tau \in C_0^\infty (M)$ is some function equal to one in a neighborhood of $K$, we obtain
\begin{align*}
	&
	\int_M
	|\nabla (v - \psi_j)|^p
	dV
	=
	\int_M
	|\nabla ((1 - \tau)(v - v_j))|^p
	dV
	\\
	&
	\quad
	{}
	\le
	2^p
	\int_{
		\operatorname{supp} \tau
	}
	|\nabla \tau(v - v_j)|^p
	dV
	+
	2^p
	\int_M
	|(1 - \tau) \nabla (v - v_j)|^p
	dV
	\to
	0
	\quad
	\mbox{as } j \to \infty,
\end{align*}
whence we again arrive at relation~\eqref{pt2.1.4} from which~\eqref{pt2.1.5} follows.

It remains to show that condition~\eqref{t2.1.1} implies the existence of a solution of problem~\eqref{1.1}--\eqref{1.3}.
Let~\eqref{t2.1.1} be valid for some $w \in \mathfrak{H}$.
We take pre-compact Lipschitz domains $\Omega_i \subset \Omega_{i+1}$, $i = 1,2,\ldots$, 
whose union coincides with the entire manifold $M$.
Consider the functions
$
	\varphi_i
	\in 
	{\stackrel{\rm \scriptscriptstyle o}{W}\!\!{}_p^1 (M)}
$ 
such that
$$
	\left.
		\varphi_i
	\right|_{
		\overline{\Omega}_i
		\cap
		\partial M
	}
	=
	h - w
	\quad
	\mbox{and}
	\quad
	\int_M
	|\nabla \varphi_i|^p
	dV
	<
	\operatorname{cap}_{h - w} 
	(
		\overline{\Omega}_i
		\cap
		\partial M
	)
	+
	\frac{1}{2^i},
	\quad
	i = 1,2,\ldots.
$$
In view of~\eqref{t2.1.1}, the sequence $\nabla \varphi_i$, $i = 1,2,\ldots$, is bounded in the space $L_p (M)$. Hence, there exists a subsequence $\nabla \varphi_{i_j}$, $j = 1,2,\ldots$, of this sequence that weakly converges in $L_p (M)$ to some vector-function ${\mathbf r} \in L_p (M)$. 
As above, we denote by $R_m$ the convex hull of the set $\{ \varphi_{i_j} \}_{j \ge m}$.
By Mazur's theorem, there exists a sequence $r_m \in R_m$, $m = 1,2,\ldots$, such that~\eqref{pt2.1.1} holds.
Since the functional
$$
	\varphi 
	\mapsto
	\int_M
	|\nabla \varphi_i|^p
	dV,
	\quad
	\varphi
	\in
	{\stackrel{\rm \scriptscriptstyle o}{W}\!\!{}_p^1 (M)},
$$
is convex, we obtain
\begin{equation}
	\int_M
	|\nabla r_m|^p
	dV
	<
	\operatorname{cap}_{h - w} 
	(
		\partial M
	)
	+
	\frac{1}{2^m},
	\quad
	m = 1,2,\ldots.
	\label{pt2.1.8}
\end{equation}
Also, it can be seen that
\begin{equation}
	\left.
		r_m
	\right|_{
		\overline{\Omega}_m
		\cap
		\partial M
	}
	=
	h - w,
	\quad
	m = 1,2,\ldots.
	\label{pt2.1.9}
\end{equation}
One can assume without loss of generality that $\Omega_1 \cap \partial M \ne \emptyset$.
Thus, we have
$$
	\int_{\Omega_1}
	|\varphi|^p
	dV
	\le
	C
	\int_{\Omega_1}
	|\nabla \varphi|^p
	dV
$$
for all
$
	\varphi 
	\in 
	{\stackrel{\rm \scriptscriptstyle o}{W}\!\!{}_p^1 
		(
			\Omega_1, 
			\overline{\Omega}_1 \cap \partial M
		)
	},
$
where the constant $C > 0$ does not depend on $\varphi$.
In particular,
$$
	\int_{\Omega_1}
	|r_i - r_j|^p
	dV
	\le
	C
	\int_{\Omega_1}
	|\nabla (r_i - r_j)|^p
	dV
$$
for all $i, j = 1,2,\ldots$, whence it follows that the sequence $r_i$, $i = 1,2,\ldots$, is fundamental in $L_p (\Omega_1)$.
Lemma~\ref{l2.1} implies that this sequence is also fundamental in $W_p^1 (G)$ for any pre-compact Lipschitz domain $G \subset M$. Let us denote by $u_1$ the limit of this sequence. In view of~\eqref{pt2.1.8} and~\eqref{pt2.1.9}, we obtain
\begin{equation}
	\int_M
	|\nabla u_1|^p
	dV
	<
	\operatorname{cap}_{h - w} 
	(
		\partial M
	)
	\label{pt2.1.11}
\end{equation}
and
\begin{equation}
	\left.
		u_1
	\right|_{
		\partial M
	}
	=
	h - w.
	\label{pt2.1.10}
\end{equation}

Let us construct a solution of problem~\eqref{1.1}--\eqref{1.3}.
This time we take a sequence of functions
$
	\varphi_i
	\in
	{C_0^\infty (M \setminus \partial M)},
$ 
$i = 1,2,\ldots$,
such that
$$
	\int_M
	|\nabla (u_1 + w - \varphi_i)|^p
	dV
	\to
	\inf_{
		\varphi
		\in
		C_0^\infty (M \setminus \partial M)
	}
	\int_M
	|\nabla (u_1 + w - \varphi)|^p
	dV
	\quad
	\mbox{as }
	i \to \infty.
$$
By~\eqref{pt2.1.11}, the sequence $\nabla \varphi_i$, $i = 1,2,\ldots$, is bounded in $L_p (\Omega)$. Thus, it has a subsequence $\nabla \varphi_{i_j}$, $j = 1,2,\ldots$, that weakly converges in $L_p (M)$ to some vector-function ${\mathbf r} \in L_p (M)$.
According to Mazur's theorem, there exists a sequence $r_m \in R_m$, $m = 1,2,\ldots$, satisfying relation~\eqref{pt2.1.1}. Since $r_m \in C_0^\infty (M \setminus \partial M)$, $m = 1,2,\ldots$, this sequence is fundamental in $W_p^1 (G)$ for any pre-compact domain $G \subset M$. Denoting by $u_0$ the limit of this sequence, we have
$$
	\left.
		u_0
	\right|_{
		\partial M
	}
	=
	0
	\quad
	\mbox{and}
	\quad
	\int_M
	|\nabla (u_1 + w - u_0)|^p
	dV
	=
	\inf_{
		\varphi
		\in
		C_0^\infty (M \setminus \partial M)
	}
	\int_M
	|\nabla (u_1 + w - \varphi)|^p
	dV.
$$
To complete the proof, it remains to note that, in view of~\eqref{pt2.1.10} and the variational principle, the function $u = u_1 + w - u_0$ is a solution of~\eqref{1.1}--\eqref{1.3}.
\end{proof}



\begin{thebibliography}{99}

\bibitem{BK}
V.\,V.~Brovkin, A.\,A.~Kon'kov,
Existence of solutions to the second boundary-value problem
for the $p$-Laplacian on Riemannian manifolds,
Math. Notes
109:2 (2021) 171--183.

\bibitem{GCh}
R.\,R.~Gadyl'shin, G.\,A.~Chechkin, 
A boundary value problem for the Laplacian with rapidly changing type of boundary conditions in a multi-dimensional domain, 
Siberian Math. J. 40:2 (1999) 229--244.

\bibitem{Grigiryan}
A.\,A.~Grigor'yan, 
Dimension of spaces of harmonic functions, 
Math. Notes 48:5 (1990) 1114--1118.

\bibitem{KonkovDiffUr}
A.\,A.~Kon'kov,
On the solution space of elliptic equations on Riemannian manifolds,
Differential Equations 31:5 (1995) 745--752.

\bibitem{KonkovMatSb}
A.\,A. Kon'kov,
On the dimension of the solution space of elliptic systems in unbounded domains,
Sbornik Mathematics 1995, 80:2, 411--434.

\bibitem{KL2012}
S.\,A. Korolkov, A.\,G.~Losev,
Generalized harmonic functions of Riemannian manifolds with ends,
Math. Z. 272:1--2 (2012) 459--472.

\bibitem{LM2019}
A.\,G.~Losev, E.\,A.~Mazepa,
On solvability of the boundary value problems for harmonic function on noncompact Riemannian manifolds,
Issues Anal. 8(26):3 (2019) 73--82.

\bibitem{Kudryavtsev}
L.\,D.~Kudrjavcev,
Solution of the first boundary value problem for self-adjoint elliptic equations in the case of an unbounded region. Izv. Akad. Nauk SSSR Ser. Mat. 31 (1967) 1179–1199 (Russian). 

\bibitem{Landkof}
N.\,S.~Landkov,
Foundations of Modern Potential Theory.
Springer-Verlag, Berlin 1972.

\bibitem{LU}
O.\,A.~Ladyzhenskaya, N.~N.~Ural'tseva,
Linear and quasilinear elliptic equations,
Academic Press, New York-London, 1968.

\bibitem{Mazya}
V.G. Maz'ya,
Sobolev spaces,
Springer Ser. Soviet Math., Springer-Verlag, Berlin 1985.

\bibitem{MPPMA2011}
V.\,G.~Maz'ya, S.\,V.~Poborchi,
Existence and uniqueness of an energy solution to the Dirichlet problem for the Laplace equation in the exterior of a multi-dimensional paraboloid,
J. Math. Sci. 172:4 (2011) 532--554.

\end{thebibliography}
\end{document}